\documentclass[12pt,onecolumn]{IEEEtran}

 \IEEEoverridecommandlockouts
 
\usepackage{url}
\usepackage{amsmath}
\usepackage{amsmath,amssymb}
\usepackage{amsthm,thmtools}
\usepackage{amsfonts} 

\usepackage[dvips]{epsfig}      % ******* for figures ***************
\usepackage{graphicx}
\usepackage{color}                  %******** for colored remarks********
\usepackage{enumitem}
\usepackage{verbatim}
\usepackage[bottom]{footmisc}

\usepackage{amssymb}
\usepackage{amsbsy}
\usepackage{amsmath}
\usepackage{setspace}
\usepackage{url}

 \newcommand{\Int}{\operatorname{int}}
 \newcommand{\real}{\operatorname{Re}}
 \newcommand{\imag}{\operatorname{Im}}

\newcommand{\Clos}{\operatorname{clos}} 
\newcommand{\Span}{\operatorname{span}}
\newcommand{\Codim}{\operatorname{codim}}
\newcommand{\Dim}{\operatorname{dim}}
\newcommand{\Rank}{\operatorname{rank}}

\newcommand*\diff{\mathop{}\!\mathrm{d}}

\declaretheoremstyle[bodyfont=\normalfont]{normalfont}

\declaretheorem[name={Example},qed={\lower-0.3ex\hbox{$\square$}} ] {Example}

\declaretheorem[name={Definition}  ] {Definition}
\declaretheorem[name={Theorem}, style=normalfont] {Theorem}
\declaretheorem[name={Lemma}, style=normalfont ] {Lemma}
\declaretheorem[name={Remark}  ] {Remark}
\declaretheorem[name={Corollary} ,style=normalfont ] {Corollary}

\newcommand {\R}{\mathbb R}

\newcommand {\C}{\mathbb C}

\newcommand{\be}{\begin{equation}}
\newcommand{\ee}{\end{equation}}

\renewcommand{\thesection}{\arabic{section}}
\usepackage{lineno}
\begin{document}
%
%\onecolumn  % add this line after \begin{document} but before \titlepage

%%\doublespace

\title{Discrete-time $k$-positive linear systems\thanks{The research of MM is supported in part by  research grants from   the Israel Science Foundation and the US-Israel Binational Science Foundation.}}
%%%%%%%%%%%%%%%%%%%%%%%%%%%%%%%%%%%%%%%%%%%%%%%%%%%%%
 \author{Rola Alseidi,    Michael Margaliot, and  J{\"u}rgen Garloff\thanks{
 		\IEEEcompsocthanksitem
 		Rola Alseidi and J{\"u}rgen Garloff  are with the 
 		Department of Mathematics and Statistics, University of Konstanz, Germany. J{\"u}rgen Garloff is also with the Institute for Applied Research at the University of Applied Sciences / HTWG Konstanz.
 		\IEEEcompsocthanksitem
 		Michael Margaliot (Corresponding Author) is  with the School of Electrical  Engineering,
 		%and the Sagol School of Neuroscience, 
 		Tel-Aviv University, Tel-Aviv~69978, Israel.
 		E-mail: \texttt{michaelm@eng.tau.ac.il}
 }}
 
%%%%%%%%%%%%%%%%%%%%%%%%%%%%%%%%%%%%%%%%%%%
\maketitle
%%%%%%%%%%%%%%%%%%%%%%%%%%%%%%%%%%%%%%%%%%%
 %\begin{multicols}{2} 
\begin{center} 
							
\end{center}

%%%%%%%%%%%%%%%%%%%%%%%%%%%%%%%%%%%%%%%%%%%%%%%
\begin{abstract}
%%%%%%%%%%%%%%%%%%%
 Positive systems play an important role in systems and control theory and have found many applications in multi-agent systems, neural networks, systems biology, and more. Positive systems map the nonnegative orthant to ​itself (and also the nonpositive orthant to itself). ​
 In other words, they map the set of vectors with zero sign variations to itself. In this note, discrete-time linear systems that map ​
 the set of vectors with up to~$k-1$ sign variations to itself are introduced. For the special case~$k=1$ these reduce to 
   discrete-time positive linear systems. ​Properties of these systems are analyzed using tools from  the theory of sign-regular  matrices. In particular, it is shown that almost every solution of such systems converges to the set of vectors with up to~$k-1$ sign variations. Also, the operation of such systems on $k$-dimensional parallelotopes are studied.

 \end{abstract}

%%%%%%%%%%%%%%%%%%%%%%%%%%%%%%%%%%%%%%%%%%%%%%%

%%%%%%%%%%%%%%%%%%%%%%%%%%%%%%%%%%%%%%%%%%%%%%%

%%%%%%%%%%%%%%%%%%%%%%%%%%%%%%%%%%%%%%%%%%%%%%%

\section{Introduction} 
%%%%%%%%%%%%%%%%%%%%%%%%%%%%%%%%%%%%%%%%%%%%%%%%%%%%%%
	For two vectors~$a, b\in\R^n$, we write~$b\leq a$ if~$b_\ell\leq a_\ell$ for all~$\ell \in\{1,\dots,n\}$. 
%%%
  Inequalities between matrices are also understood as entry-wise.
	
	Consider the discrete-time~(DT) 
linear time-varying~(LTV) 
system
\be\label{eq:dts}
x(i+1)=A(i)x(i),\quad x(0)=x_0\in\R^n.
\ee
%%%
Let~$x(i,x_0)$ denote the 
solution of~\eqref{eq:dts} at time~$i$. 
  The LTV~\eqref{eq:dts}  is called \emph{positive}
 if~$A(i)\geq 0$
for all~$i\geq 0$. 
Then clearly 
\be 
b\leq a \implies x(i,b)\leq x(i,a) \text{ for all } i\geq 0 . 
\ee
In particular,
\[
0\leq a \implies 0\leq x(i,a) \text{ for all } i\geq 0 ,
\]
i.e.,  
a positive system maps the nonnegative orthant 
\[
\R^n_+:=\{x\in\R^n:x_i\geq 0 \text{ for all } i\}
\]
 to itself (and also~$\R^n_-$ to itself). 
The system is called \emph{strongly positive} if it maps~$\R^n_+\setminus\{0\}$ to~$\Int(\R^n_+)$ (the  
interior of~$\R^n_+$).

Positive systems appear naturally when the state-variables represent quantities
that can only take nonnegative values,  e.g.,  probabilities, concentrations of molecules, densities of particles, etc. 
Positive LTVs play an important role in linear systems and control theory, see, e.g.,~\cite{farina2000,posi-tutorial}, and via \emph{differential analysis}~\cite{forni2016, LOHMILLER1998683},
 also   in the analysis of nonlinear systems. 
To explain this, consider the \emph{nonlinear}
 time-varying  system
\be\label{eq:nonline}
x(i+1)=f(i,x(i)),
\ee
and suppose that its trajectories evolve on a convex state-space~$\Omega \subseteq \R^n$,
and that~$f$ is~$C^1$ with respect to~$x$.
For~$y\in\Omega$, let~$x(i,y)$ denote the solution of~\eqref{eq:nonline} at time~$i$ for~$x(0)=y$. 
Pick~$a,b\in\Omega$ and let
\[
z(i):=x(i,a)-x(i,b),
\]
that is, the difference at time~$i$ between the trajectories emanating from~$a$ and from~$b$ at 
time zero. Then
\begin{eqnarray}\label{eq:vari}
z(i+1)&=f(i,x(i,a))-f(i,x(i,b))=J^{ab} (i)z(i) \\
  \text{with }  J^{ab} (i)    & := \int_0^1 \frac{\partial}{\partial r} f(i, r x(i,a) +(1-r) x(i,b) ) \diff r \nonumber.
      \end{eqnarray}

If~$J^{ab}(i)\geq0$ for all~$a,b\in\Omega$ and all~$i\geq0$, then the \emph{variational system}~\eqref{eq:vari} is a positive~LTV, and this has important consequences for the behavior of~\eqref{eq:nonline}. 
Roughly speaking, almost every  bounded trajectory  of a smooth strongly positive   system
converges  to a periodic 
trajectory (a cycle)~\cite{Typical_mono_ds}. 
This is quite different from  the behavior in the
continuous-time case, where 
almost every  bounded trajectory  of the nonlinear system 
converges  to the set of equilibria~\cite{hlsmith}.

The dynamics of  a DT positive LTV  maps the set of vectors with zero sign variations
to itself. A natural question is: what systems map the set of vectors with up to~$k-1$
sign variations to itself? We call such a system a \emph{DT~$k$-positive system}. 
Then a~$1$-positive system is just a positive system, but for~$k>1$
the system may be~$k$-positive yet not positive.

 Continuous-time~(CT)   $k$-positive systems have  been recently
defined and analyzed in~\cite{eyal_k_posi}.   In the CT and time-invariant 
case, i.e.,~$\dot x(t)=Ax(t)$,
 the \emph{matrix exponential}
 of~$A$ should satisfy for all time a property  called strict
 sign-regularity of order~$k$, for the definition see the next paragraph, 
and this can be mapped to simple to check sign
conditions on~$A$ itself~\cite{eyal_k_posi}. 
In the DT case studied here, the matrix~$A$ itself 
must have this property, and verifying  this is nontrivial.

A matrix~$A\in \R^{n\times m}$
  is called  \emph{sign-regular of order~$k$}  (denoted by~$SR_k$)
if all its minors of order~$k$ are non-negative or all are non-positive. For example, if all the entries of~$A$ are non-negative then it is~$SR_1$.	A matrix is  called \emph{strictly sign-regular of order $k$} (denoted by~$SSR_k$) if it is~$SR_k$, and all the minors of order~$k$ are non-zero. In other words, all minors of order $k$ are non-zero and have the same sign.\footnote{We note that the terminology in this field is not uniform and some authors refer to such matrices as  sign-consistent of order $k$.}

For example, consider the 
matrix 
$$A:=\begin{bmatrix}  1& 2 & 0 &0 \\
0&1&1&0 \\ 0 & 0&2 &0.1\\1&0&0 &2
\end{bmatrix}.$$
This matrix is~$SR_1$ (but not~$SSR_1$ as some entries are zero). It has both positive and negative~$2$-minors (e.g., $\det(\begin{bmatrix} 1&2\\0&1\end{bmatrix})=1$,
$\det(\begin{bmatrix} 1&2\\1&0\end{bmatrix})=-2$),
so it is not~$SR_2$. All its~$3$-minors are positive, so it is~$SSR_3$, and~$\det(A)>0$, so it is~$SSR_4$. \\

 After the first consideration of $SR_k$ matrices  in \cite{karlin1968total}, these  matrices   have been the subject of only a few studies. In Ref.~\cite{rola_spect}, the authors  analyze
	 the spectral properties of nonsingular matrices that are~$SSR_k$ for a specific value of~$k$. 
	These results are extended to  matrices that are~$SSR_k$ for several values of~$k$, for example for all odd~$k$.

To refer to the common sign of the minors of  order $k$,  we introduce the signature  $\epsilon \in \{-1,1\}$.  A matrix $A \in \mathbb{R}^{n \times m}$ is called 
\textit{[strictly] sign-regular} $([S]SR)$ if it is $[S]SR_k$ for all~$k=1,\dots,\min\{n,m\}$.
 
%%

%%%

The most important examples of~$SR$ [$SSR$] matrices
	  are the totally nonnegative~\emph{(TN)} [totally positive~\emph{(TP)}]
 matrices, that is,
matrices  with all minors nonnegative [positive].
Such matrices have applications in  numerous  fields including 
approximation theory, combinatorics, probability theory,  computer aided geometric design, differential and integral equations, and more~\cite{total_book,gk_book,karlin1968total,pinkus}. 

A very important property  of~$SSR$ matrices is  that multiplying a vector $x$ by such a matrix cannot  increase the number of sign variations in $x$~\cite{gk_book}. 
To explain this \emph{variation diminishing property}~(VDP), we introduce some notation.
For~$y \in \mathbb{R}^n$, 
let~$s^-(y)$   denote  the number of sign variations
in~$y$ after deleting all its zero entries,
and let~$s^+(y)$   denote  the maximal possible number of sign variations
in~$y$ after each zero entry is replaced by either~$+1$ or~$-1$. 
For example, for~$n=4$ and~$y=\begin{bmatrix} 1& -1 & 0  &  -\pi  \end{bmatrix}^T$ (where the superscript~$T$ denotes transposition), we have~$s^-(y)=1$ and~$s^+(y)=3$. 
Obviously,
\be\label{eq:smsp}
%%%
0\leq s^-(y) \leq s^+(y)\leq n-1 \text{ for all } y\in\R^n.
\ee
% Let~$	\W:=\{y\in\R^n:s^-(y)=s^+(y)\}$, 

\begin{comment} ** I DONT THINK THAT THIS PART IS RELEVANT FOR THE CURRENT PAPER *** 
%%%%%%%%%%%%
The \emph{cyclic number of sign variations} in $y \in \mathbb{R}^n$  is defined similarly:  
\[
s^-_c(y):=\max\limits_{i\in \{1,\dots,n\}}s^-([y_i,\dots,y_n,y_1,\dots,y_i]^T)
\] 
and
\[
s^+_c(y):=\max\limits_{i\in \{1,\dots,n\}}s^+([y_i,\dots,y_n,y_1,\dots,y_i]^T),
\]  
where here if the first entry~$y_i$ is replaced by~$1$ or~$-1$  
then the last entry is replaced by the same value. 
Theorem 2 in \cite{ben2018dynamical} provides a necessary and sufficient condition for a rectangular  and nonsingular matrix to satisfy the strong cyclic variation diminishing  property, see Theorem \ref{Maichel cyclic}.\\
%%%%%
 \end{comment}

The first important results on the VDP of matrices were obtained by Fekete and P{\'o}lya~\cite{fekete1912} and Schoenberg~\cite{Schoenberg1930}. Later on, Gantmacher and Krein \cite[Chapter V]{gk_book} elaborated rather completely the various forms of VDPs and worked out the spectral properties of~$SR$ matrices. Two important
examples of such~VDPs are: if~$A\in\mathbb{R}^{n\times m}$ ($m \leq n$) 
is~$SR$ and of rank~$m$ then
\begin{equation*}
s^-(Ax)\leq s^-(x) \text{ for all } x \in \mathbb{R}^m,
\end{equation*}
whereas if $A$ is   $SSR$ then  
\begin{equation*}
s^+(Ax)\leq s^-(x) \text{ for all } x \in \mathbb{R}^m \setminus \{0\}.
\end{equation*} 
Note that combining this with~\eqref{eq:smsp}
implies that both~$s^-(x(i))$ and~$s^+(x(i))$ can be used as an integer-valued Lyapunov function for the system~$  x(i+1)=Ax(i)$.

For~$k\in\{1,\dots,n\}$,  let
 \begin{align}\label{eq:defpk0}
 P^k_- &:=\{ z\in\R^n : s^-(z)\leq k-1 \},\nonumber\\
  P_+^k &:=\{ z\in\R^n : s^+(z)\leq k-1 \}.
	\end{align}
Then 
positive systems map the set~$P^1_-$ to~$P^1_-$, whereas strongly positive systems map~$P^1_-$ to~$P^1_+$. This naturally leads to the question: which linear systems map~$P^k_-$ to~$P^k_-$ 
and which map~$P^k_-$ to~$P^k_+$?
In this  this paper,  we define and analyze such systems, called    DT $k$-positive linear systems. We show that such systems have interesting dynamical properties that generalize the properties of positive systems. % Second, we derive new conditions for verifying that a rectangular  matrix is~$SSR_k$. Note that Pe{\~{n}}a~\cite[Theorem 2.2]{Pena}
%derived an elegant necessary and sufficient condition for a
% matrix~$A \in \mathbb{R}^{n\times k}$, with $n>k$, to be~$SSR_k$.
%However, this does not apply to the case of a rectangular matrix that is the relevant case in applications to dynamical systems. 
%Also, there exist criteria for checking if~$A\in\R^{n\times m}$
%is~$SSR$ 
%(see e.g.~\cite[p.~261]{gk_book}), but here we are interested in the much weaker property of~$SSR_k$ for some value of~$k$. 
%%%%%%%%%%%%%%%%%%%%%%%%%%%%%%%%%%%%%%%%%%%%%%%

The remainder of this paper 
 is organized as follows. In the next section, we review
 notations, definitions, and basic properties   that will be used later on.
%%%%%%%%%%%%%%%%%%%%%%%%%%%%%%%%%%%%%%%%%%%%%%%%%% 
Section~\ref{sec:main1} defines   DT $k$-positive linear systems and analyzes their properties.
 %%%%%%%%%%%%%%%%%%%%%%%%%%%%%%%%%%%%%%%
The final section concludes. In passing we note that our results are part of a growing
body of 
  research   on the applications of sign-regularity (and, in particular, total positivity) to dynamical systems~\cite{rola_spect,CTPDS,rami_osci,margaliot2019revisiting,schwarz1970,eyal_k_posi}.

%%%%%%%%%%%%%%%%%%%%%%%%%%%%%%%%%%%%%%%%%%%%%%%
\section{Preliminaries}
%%%%%%%%%%%%%%%%%%%%%%%%%%%%%%%%%%%%%%%%%%%%%%%
We briefly review several known results that will be used later on. 
%%%%%%%%%%%%%%%%%%%%
\subsection{Basic notation and definitions}
%%%%%%%%%%%%%%%%%%%%%%%%%%%%%%%%%%%%%
  For an integer~$n\geq 1$ and~$k\in\{1,\dots,n\}$,
	let~$Q_{k,n}$ denote the set of all strictly increasing sequences of~$k$ integers chosen from $\{1,\dots,n\}$.
	For example,~$Q_{2,3}=\{  12,13,23    \}$. 
	
	\begin{comment}
	A measure of the gap in an index sequence $\alpha= (\alpha_1,\alpha_2,\dots,\alpha_k)\in Q_{k,n}$ is the \emph{dispersion} of~$\alpha$, defined by
	\begin{align}\label{eq:dispe}
	{d}(\alpha)&:=\sum_{j=2}^{k} (  \alpha_j-\alpha_{j-1}-1) \nonumber \\
	&=\alpha_k-\alpha_1-|\alpha|+1,
	\end{align}
	where~$|\alpha|$ denotes  the cardinality of~$\alpha$.
	Thus,~$d(\alpha)$ is the number of integers between~$\alpha_1$ and~$\alpha_k$ that are not included in~$\alpha$. 
	For example~$d(1, 2,\dots, k)= 0   $ and~$d(1,3,5 )=  2   $.
	If ${d}(\alpha)=0$, we call $\alpha$ \textit{contiguous}. 
	\end{comment}
	
	For  $A \in \mathbb {R}^{n\times m}$, $\alpha \in Q_{k,n}$, and~$\beta \in Q_{j,m}$, we denote the submatrix of~$A$ lying in the rows indexed by~$\alpha$ and columns indexed by~$\beta$ by $A[\alpha,\beta] $.
	Thus,~$A[\alpha,\beta]  \in \R^{k\times j}$. If~$k=j$ then we set
	\[
	A(\alpha|\beta):=\det (A[\alpha,\beta] ),
	\]
	that is, the minor corresponding to the rows indexed by~$\alpha$ and columns indexed by~$\beta$.
 We often suppress the brackets associated with  an index sequence if we enumerate its entries explicitly. 
	
	\begin{comment}
	A submatrix $A[\alpha|\beta]$ or a minor $A(\alpha|\beta)$ is called \textit{row contiguous} if $\alpha$ is contiguous,
	it is called \textit{column contiguous} if~$\beta$ is contiguous. If it is both row and column contiguous, we simply say that it is \textit{contiguous}.  
 \end{comment}

%%%%%%%%%%%%%%%%%%%%%%%%%%%%%%%%%%%%%%%%%%%%
\subsection{Multiplicative compound}
%%%%%%%%%%%%%%%%%%%%%%%%%%%%%%%%%%%%%%%%%%%%
Let~$A\in \R^{ n\times m}$. For any~$k\in\{1,\dots,\min\{n,m\}\}$, the~\emph{$k$th multiplicative compound} of~$A$ is the~$\binom{n}{k}\times \binom{m}{k}$ matrix that includes all the minors of order $k$ of~$A$ 
organized in lexicographic order.
For example, if~$A\in\R^{3\times 3}$ then
\[
A^{(2)}= \begin{bmatrix} 
A(12|12) & A(12|13) & A(12|23) \\
A(13|12) & A(13|13) & A(13|23) \\
A(23|12) & A(23|13) & A(23|23) 
%%%
\end{bmatrix}.
\]
%%%
 Note that~$A^{(1)}=A$ and that if~$m=n$ then~$A^{(n)}=\det(A)$. 
Note also that~$A$ is~$SSR_k$ [$SR_k$]
if either~$A^{(k)}>0$ or~$A^{(k)}<0$ [either~$A^{(k)}\geq 0$ or~$A^{(k)}\leq 0$]. 

The Cauchy-Binet Formula~\cite[Theorem 1.1.1]{total_book} provides an expression for the minors of the product 
 of two matrices.
Pick~$A \in \mathbb{R}^{n\times p}$ and $B \in \mathbb{R}^{p \times m}$. Let~$C:=AB$.  Pick~$k\in \{1,\dots, \min \{n,p,m\}\}$,
$\alpha \in Q_{k,n}$, and~$\beta \in Q_{k,m}$. Then
	\begin{eqnarray}\label{bienteq}
	C(\alpha|\beta)=\sum_{\gamma \in Q_{k,p}}A(\alpha|\gamma)B(\gamma|\beta).
	\end{eqnarray} 
For~$n=p=m$ and~$k=n$ this reduces to the familiar formula~$\det(AB)=\det(A)\det(B)$. 	
Note that~\eqref{bienteq} implies that 
\be\label{eq:mulcom}
%%%%%%%%%%%%%%%%%%%%%%%%%%
(AB)^{(k)}=A^{(k)} B ^{(k)}
\ee
for all~$k\in  \{1,\dots, \min \{n,p,m\}\}$.
This justifies the term   multiplicative compound. 

 %%%%%%%%%%%%%%%%%%%%%%%%%%%%%%%%%%%%%%%%%%%%%%%%%%%%%%
\subsection{Sets of vectors with sign variations}
%%%%%%%%%%%%%%%%%%%%%%%%%%%%%%%%%%%%%%%%%%%%%%%%%%%%%%%%%%%
	%%
	Consider the sets defined in~\eqref{eq:defpk0}. It is straightforward to show using the 
	definitions of~$s^-$, $s^ +$ and~\eqref{eq:smsp}
	that~$P^k_-$ is closed and that
	\[
	P_+^k=\Int(P_-^k).
	\]

 It is clear that
\be\label{eq:p1}
P^1_-=\R^n_+ \cup \R^n_- , \quad  P^1_+=\Int (\R^n_+ )\cup  \Int( \R^n_-) .
\ee
Also, the sets are    nested, as 
\begin{align}\label{eq:psdv}
%%%
%%%
   P^1_- \subset P^2_- \subset \dots \subset P^n_-=\R^n,\nonumber\\
   P^1_+ \subset P^2_+ \subset \dots \subset P^n_+=\R^n.
\end{align}
If~$x\in P^k_-$ then~$\alpha x \in P^k_-$ for all~$\alpha\in\R$,
and if~$x\in P^k_+$ then~$\beta x \in P^k_+$ for all~$\beta \in\R\setminus\{0\}$,
so both~$P^k_-$ and~$P^k_+ \cup\{0\}$ are cones.
Yet,   in general~$P^k_-$ and~$P^k_+$ are \emph{not}
convex sets. For example, for~$n=2$ and the vectors~$x:=\begin{bmatrix} 2& 0\end{bmatrix}^T$,
 $y:=\begin{bmatrix} 0& -2 \end{bmatrix}^T$, we have~$x,y\in P^1_-$ yet
$\frac{x}{2}+\frac{y}{2} =\begin{bmatrix} 1&  -1\end{bmatrix}^T \not \in P^1_-$.

Recall that a set $C \subseteq \mathbb{R}^n $ is called \textit{a cone of rank~$k$}~\cite{pls_sobolev} if
\begin{enumerate}[label=(\roman*)]
	\item $C$ is closed;
	\item $x\in C$ implies that $\alpha x \in C$ for all $\alpha \in \mathbb{R}$;
	\item $C$ contains a linear subspace of dimension $k$ and no linear subspace of higher dimension.
\end{enumerate}
For example,  $\mathbb{R}_+^2 \cup \mathbb{R}_-^2$ (and more generally, $\mathbb{R}_+^n \cup \mathbb{R}_-^n$~\cite{fuscoPF}) is a cone of rank~1. A cone~$C$ of rank~$k$ is called \textit{solid} if its interior is non empty, and $k$-$solid$ if there is a linear subspace~$W$ of dimension~$k$ such that $W\setminus \{0 \} \subseteq \Int(C)$; $k$-solid cones are useful
 in the analysis of dynamical systems~\cite{hordercones17, generic_kcone,p_domi,sanchez2009cones}.
Roughly speaking, if a trajectory of the system is confined to an invariant set~$C$ 
that is a $k$-solid cone then the trajectory can be projected onto a~$k$-dimensional subspace contained in~$C$.
If this projection is one-to-one then the trajectory is topologically conjugate to a trajectory of a~$k$-dimensional dynamical system.

It was shown in~\cite{eyal_k_posi} (see also~\cite{pls_sobolev}) that for
any~$k\in \{1,\dots,n\}$,
the set~$P^k_-$
is a~$k$-solid cone, and that its complement
\[
(P^k_-)^c:=\Clos(\R^n\setminus P^k_-)
\]
is an~$(n-k)$-solid cone. 
This implies, in particular, that there exists a~$k$-dimensional subspace~$W^k$ such that~$W^k\subseteq P^k_-$, and that there is no~$(k+1)$-dimensional subspace contained in~$P^k_-$. For example, 
let~$e^i \in\R^n$ denote the vector with all entries zero,
except for entry~$i$ that is one. 
Then the~$k$-dimensional subspace
$\Span\{e^1,\dots,e^k\}$ 
is contained in~$P^k_-$.

%%%%%%%%%%%%%%%%%%%%%%%%%%%%%%%%%%%%%%%%%
\subsection{Linear mappings that preserve the number of sign variations in a vector}
%%%%%%%%%%%%%%%%%%%%%%%%%%%%%%%%%%%%%%%%%%%%%%%%%%%%%%%%%%%%%%%%%%%%%%%%%%%%%%%%%%%%
Recall that a matrix~$A\in \R^{n\times m}$ is termed~$SR_k$ if all its minors of order~$k$
are either all nonnegative or all nonpositive and it is called $SSR_k$ if 
all its minors of order~$k$ are all positive or all negative,  see the example in the Introduction.

Let~$A\in \R^{n\times n}$ be a nonsingular matrix. Pick~$k\in\{1,\dots,n\}$. 
It was shown in~\cite{CTPDS} that~$A$ maps~$P^k_- \setminus\{0\}$ to~$P^k_+$ if and only if~$A$ is~$SSR_k$, 
and a continuity argument~\cite{eyal_k_posi} implies that~$A$
 maps~$P^k_-$ to~$P^k_-$ if and only if
$A$ is~$SR_k$. For example, the nonsingular matrix 
$$A:=\begin{bmatrix}  10& 4& 1\\
 1& 3& 1 \\  2& 4& 6
\end{bmatrix}$$  is not~$SR_2$,  as it has both positive and negative minors of order~$2$
(e.g.,  $A(1,2|1,2)=26$ and $A(2,3|1,2)=-2$), so it does not map~$P^2_-$ to itself. 
Indeed, for~$x=\begin{bmatrix} 19&-6&-2\end{bmatrix}^T$, we have~$x \in P^2_-$ 
and~$Ax= \begin{bmatrix} 164&-1&2\end{bmatrix}^T \not \in P^2_-$.

We can now introduce and analyze a new class of DT linear systems.

 %%%%%%%%%%%%%%%%%%%%%%%%%%%%%%%%%%%%%%%%%%%%%%%%%%%%%%%%%%%%%
\section{Discrete-time~$k$-positive linear systems}\label{sec:main1}
%%%%%%%%%%%%%%%%%%%%%%%%%%%%%%%%%%%%%%%%%%%%%%%%%%%%%%%%%%%%%
%%%%%%%%%%%%%%%%%%%%%%%%%%%%%%%%%%%%%%%%%%%%%%%%%%%%%%%%%%%%%

%%%%%%%%%%%%%%%%%%%%%%%%%%%%%%%%%%%%%
\begin{Definition}
%%%%%%%%%%%%%%%%%%%%%%%%%%%%%%
Consider the DT LTV~\eqref{eq:dts} 
with
every matrix~$A(i)$   nonsingular.
The system is called  a~\emph{$k$-positive system} if 
it maps~$P^k_-$ to~$P^k_-$, and a \emph{strongly~$k$-positive system} if it maps~$P^k_-\setminus\{0\}$ to~$P^k_+$.
%%%%%%%%%%%%%%%%%%%%%%%%%%%%%%%%%
\end{Definition}

Note that~\eqref{eq:p1} implies that a~[strongly] $1$-positive system is 
simply a [strongly] positive system. Note also that since~$P^k_+ =\Int(P^k_-)$, both~$P^k_-$ and~$P^k_+$
are invariant sets of 
 a strongly~$k$-positive system.

	The next result follows from~\cite[Theorem~1]{CTPDS}
	and~\cite[Theorem~2]{eyal_k_posi}.
	%%%%%%%%%%%%%%%%%%%%%%%%%%%%%%%%%%%%%%
\begin{Corollary}
The system~\eqref{eq:dts}  is
a [strongly] $k$-positive system if and only if~$A(i)$ is $[S]SR_k$  for all~$i \geq 0$. 

\end{Corollary}

For the case~$k=1$ this is a generalization of [strongly] positive linear systems. For example, a system is typically defined as strongly positive if
 all the  entries of~$A(k)$ 
are positive, yet it is strongly~$1$-positive 
if all its entries are either all positive or all negative. 

Form here on we focus on the time-invariant linear system
\be\label{eq:ltis}
x({j}+1)=Ax({j}),\quad x(0)=x_0 \in \R^n,
\ee
 %%%%%%%%%%%%%%%%
where~$A$ is nonsingular and~$SSR_k$ for some~$k\in\{1,\dots,n-1\}$, leaving the time-varying case and nonlinear systems to a sequel paper.  Note that even for this LTI
case our results are new. 

 \begin{Example}\label{exa:23p}
 	%%%%%%%%%%%%%%%%
 	Consider the system~\eqref{eq:ltis}
	with~$n=4$ and
 	\be\label{eq:mata3}
 	 A: =
 	\begin{bmatrix}     
 	%%%
 	9&2&-2&1 \\
 	%%%%%			
 	3&10&1&-1 \\
 	%%%%%			
 	-4&1.5&12&4\\
 	%%%%%			
 	1&-1&2&15																%%%%%			
 	%%%%
 	\end{bmatrix}, 
 	\ee
	for all~$j\geq 0$.
 	Note that~$A$ is not~$\text{SSR}_1$ (as it has both positive and negative entries), nor~$\text{SSR}_2$ (as it has both positive and negative minors of order two, e.g., $A(1,2|1,2)=84$, $A(3,4|1,3)=-20$). All the~$16$ minors of order three are positive, and $\det(A) \not =0$,
	so~$A$ is~$\text{SSR}_3$ and nonsingular.
	Figure~\ref{fig:perorder3}
 	shows~$s^+(x(j))$ as a function of~$j$ for~$x(0)=\begin{bmatrix}
 	1& 1&-1& 1\end{bmatrix}^T$. Note that~$s^-(x(0))=2$. 
 	It may be seen that, as expected,~$s^+(x(j))\leq 2$ for all~$j\geq 0$. 
 	%%%%%%%
 	%%%%%%%%
 \end{Example}

 \begin{figure}[t]
 	\begin{center}
 		%%%%%%%%%%%%%%%%%%%%%%%[width=12cm,height=12cm]
 		\includegraphics[scale=0.6]{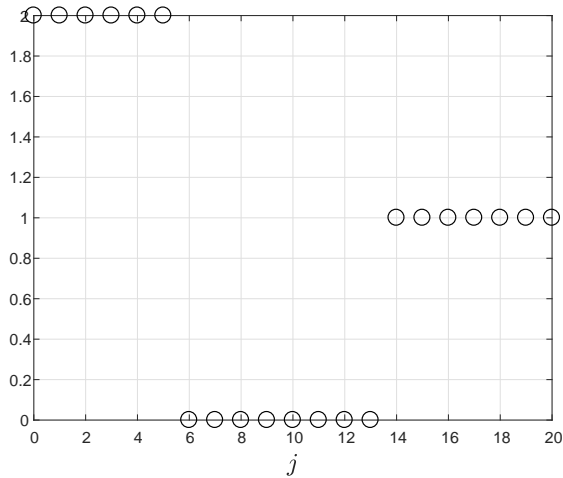}
 		\caption{ $s^+(x(j))$ as a function of~$j$
 			for the trajectory  in Example~\ref{exa:23p}.  }
 		\label{fig:perorder3}
 		%%%%%%%%%%%%%%%%%%%%%%%%%%
 	\end{center}
 \end{figure}

%%%%%%%%%%%%%%%%%%%%%%%%%%%%%%%%%%%
\subsection{ $k$-exponential separation and its implications} 
%%%%%%%%%%%%%%%%%%%%%%%%%%%%%%%%%%%
Let $C\subseteq \R^n$ be closed and a \textit{convex cone}, i.e., $x,y \in C$ implies that~$\alpha x +\beta y\in C$ for all~$\alpha, \beta\geq 0$. 
Furthermore, let $C$ be \textit{pointed}, i.e., $C \cap (-C) =\{0\}$. 
Then $C$ induces a (partial) order 
defined by~$a\leq_C b$ if~$b-a \in C$.  For example, for~$C=\R^n_+$
we have~$a\leq_C b$ if and only if~$b_i\geq a_i   $ for all~$i\in \{ 1,\dots,n\}$.
Dynamical systems whose flow  preserves such an order are called \emph{monotone}, see, e.g., the excellent monograph~\cite{hlsmith}. 

Since~$P^k_-$ and~$P^k_+$ are not convex sets, 
$k$-positive   systems are not monotone systems
in the usual sense. 
However, the fact that~$P^k_-$ is a $k$-solid cone has strong implications for the dynamics of such systems. 

The first demonstration of this is a $k$-exponential separation property of~\eqref{eq:ltis}. This is closely related to the generalization of the Perron Theorem  in~\cite{fuscoPF}, see also~\cite{pls_sobolev}, but we give a direct proof based on the spectral properties of a nonsingular~$SSR_k$ matrix, see Theorem \ref{thm:sepe} below. 
We now review these properties following the presentation in~\cite{rola_spect}. 
 
Fix a nonsingular 
matrix~$A\in\R^{n\times n}$ that is~$SSR_k$ for some~$k\in\{1,\dots,n-1\}$. Let~$\epsilon \in\{-1,1\}$ denote the 
common sign of all the minors of order $k$ . 
Denote the eigenvalues of~$A$
by~$\lambda_i$, $i=1,\dots,n$, ordered such that
\be\label{eq:eigord}
|\lambda_1|\geq |\lambda_2|\geq \dots\geq |\lambda_n|>0,
\ee
and  let
\be \label{eq:vis}  
v^1,\; v^2,\dots,v^n
\ee
 denote the corresponding eigenvectors, 
with complex conjugate  eigenvalues appearing  in consecutive pairs
(we   say, with a mild abuse of notation,
 that~$z \in \C^n$ is \emph{complex}  
if~$z\not = \bar z$, where~$\bar z$ denotes
 the complex conjugate of~$z$). 
%%%
 	We may assume that every~$v^i$ is not purely imaginary.
	Indeed, otherwise  we can replace~$v^i$ 
	by~$\imag(v^i)$ that is a real eigenvector. 
	Also, the  fact that~$A$ is real means that if~$v^i$ is complex then its real and imaginary parts can be chosen as linearly independent. 
	
	 Define a set of real vectors~$u^1,u^2,\dots,u^n \in \R^n$ by going through the~$v^i$'s as follows.
 If~$v^1$ is real then~$u^1:=v^1$ and  proceed to examine~$v^2$.
 If~$v^1$ is complex (and whence~$v^2=\bar v^1$) then~$u^1:=\real(v^1)$, $u^2:=\imag(v^1)$ and proceed to examine~$v^3$, and so on. 

Suppose that for some~$i,j$ the eigenvector~$v^i$ is real 
and~$v^j$ is complex. Then 
  is not difficult to show that
since~$A$ is real and nonsingular,
the real vectors~$v^i,\real(v^j),\imag(v^j)$ are linearly independent.

	 Note that if~$v^i,v^{i+1}\in\C^n$ is a complex conjugate pair and~$c\in \C\setminus\{0\}$ is complex then 
\[
			c   v^i+\bar c v^{i+1} =2( \real(c) \real(v^i) -\imag(c)\imag(v^i)) \in \R^n\setminus\{0\},
\]
 so  by choosing an appropriate~$c \in \C\setminus\{0\}$ we can get any nonzero real linear combination of the real vectors~$ \real(v^i)$ and~$\imag(v^i) $. 

For~$p \leq q$,
we  say that a set~$c_p,\dots,c_q\in \C$ 
 \emph{matches} the set~$v^p,\dots,v^q$ of
 consecutive eigenvectors~\eqref{eq:vis}  if the~$c_i$'s are not all zero and for every~$i$ if  the vector~$v^i$ 
is real then~$c_i$ is real,
	and if~$v^i,v^{i+1}$ is a complex conjugate pair then~$c_{i+1}=\bar c_i$.  
	In particular, this implies that~$\sum_{i=p}^q   c_i v^i \in  \R^n. $

It was shown in~\cite{rola_spect} that if~$A\in\R^{n\times n}$ is nonsingular 
and~$SSR_k$ with signature~$\epsilon$, 
then  the product~$\epsilon \lambda_1\lambda_2\dots \lambda_{k} $
is real and positive,   
\be\label{eq:akoi}
 |\lambda_k|>|\lambda_{k+1}|,
\ee
and   if~$c_1,\dots,c_k   \in \C$
   [$c_{k+1},\dots,c_n \in \C $] match   
   the eigenvectors~$v^1,\dots,v^k$ [$v^{k+1},\dots,v^n$] of~$A$, then  
\begin{align}
%%%%%%%%%%%%%%%%%%%%%%%%%%%%%%%%
					s^+(\sum_{i=1}^k c_i v^i)&\leq k-1 ,  \label{eq:suim}\\ 
%%%%
%%%%%%%%%%%%%%
					s^-(\sum_{i=k+1}^n c_i v^i)&\geq k .\label{eq:suim22}
%%%%%%
\end{align}
%%%%%%
Furthermore, let~$\{u^1,\dots,u^{n}\}$ be the set of real vectors
constructed from~$\{v^1,\dots,v^n\}$ as described above. 
Then~$u^1,\dots,u^{k}$ 
are linearly independent. In particular, if~$v^1,\dots,v^k$
are real then they are linearly independent.

\begin{Example}
%%%%%%%%%%%%%%%%%%%%%
Let
\be\label{eq:ant}
A:=\begin{bmatrix}  2& 6 & 0 &0 \\
											0&2&2&0 \\ 0 & 0&4 &2\\2&0&0 &4
\end{bmatrix}.
\ee
 It is straightforward to verify that this matrix is nonsingular, and that
 all minors of order $3$ are  positive,  
so~$A$ is~$SSR_3$ with~$\epsilon=1$. 
 Its eigenvalues   are  \footnote{All numerical values in this paper are subject to 4-digits accuracy.}

%%%%
$$\lambda_1=3+s_1 ,\;
\lambda_2= 3+\mathfrak{i}s_2 , \;
\lambda_3=  3-\mathfrak{i}s_2 ,\;
\lambda_4= 3-s_1 ,$$
where~$\mathfrak{i}^2=-1$, $s_1:=\sqrt{1+4\sqrt{3}}\approx 2.8157 $,
and~$s_2:=\sqrt{-1+4\sqrt{3}} \approx 2.4348$. 
Note that~$\lambda_1\lambda_2\lambda_3$ is real and positive, and that~$|\lambda_3|>|\lambda_4|$.
 %\begin{align*}
%v^1&=\begin{bmatrix} 0.79& 0.7071& 1.266& 1\end{bmatrix}',\\
% v^2&=\begin{bmatrix} -0.5 + 1.079 j& -0.7071 & -0.3536 - 0.763 j &  1 \end{bmatrix}',\\
%v^3&=\begin{bmatrix}  -0.5 - 1.079 j& -0.7071 & -0.3536 + 0.763 j& 
%  1\end{bmatrix}',\\
%	v^4&=\begin{bmatrix}-1.79& 0.7071& -0.5586& 1\end{bmatrix}'.
%	\end{align*}
 %%%
%%
The matrix of corresponding eigenvectors is
%%%
\begin{align*}
V&:=\begin{bmatrix} v^1 & v^2 & v^3 &v^4 \end{bmatrix}\\
&=\begin{bmatrix}
						\frac{s_1-1}{2} &  \frac{\mathfrak{i}s_2-1}{2} 
						&\frac{-(\mathfrak{i}s_2+1)}{2} &  \frac{-(s_1+1)}{2}
								\\[6pt]
						\frac{s_1^2-1}{12} &  \frac{-(1+s_2^2)}{12} & \frac{-(1+s_2^2)}{12} &  \frac{s_1^2-1}{12} \\[6pt]
						\frac{2}{s_1-1} & \frac{-2(1+\mathfrak{i}s_2)}{1+s_2^2} &\frac{2(-1+\mathfrak{i}s_2)}{1+s_2^2} &
							               \frac{-2}{s_1+1} \\[6pt] 
						%%%%%
							1&1&1&1
\end{bmatrix},
\end{align*}
%%%%
and thus
\begin{align*}
									U&:=\begin{bmatrix} u^1&  u^2  &u^3 & u^4 \end{bmatrix}\\
									&=\begin{bmatrix} v^1& \real(v^2) &\imag(v^2) & v^4 \end{bmatrix}\\
									%%%%%
									& = \begin{bmatrix}
						\frac{s_1-1}{2} &  \frac{-1}{2} & \frac{s_2}{2 } &  \frac{-(s_1+1)}{2}    \\[6pt]
						\frac{s_1^2-1}{12} & \frac{-(1+s_2^2)}{12} & 0&   \frac{s_1^2-1}{12} \\[6pt]
						\frac{2}{s_1-1} & \frac{-2}{1+s_2^2} &\frac{-2s_2}{1+s_2^2} &  \frac{-2}{s_1+1} \\[6pt]
						%%%%%
							1&1&0&1
\end{bmatrix}.
\end{align*}
Note that~$s^-(u^i)=s^+(u^i)=i-1$, $i=1,2,4$, and
$$1=s^-(u^3)<s^+(u^3)=2.$$ 
%%%%%%%%%%%%%
\end{Example}
%%%%%%%%%%%%%%

We now state the main result in this subsection. 
Let~$ ||\cdot||:\R^n\to \R_+$ denote some vector norm.

%%%%%%%%%%%%%%%%%%%%%%%%%%%
\begin{Theorem}\label{thm:sepe}
%%%%%%%%%%%%%%%%%%%%%%%%%%%%%%%%%%%%%%%%%%%
Suppose that~$A\in\R^{n\times n}$ is 
nonsingular and~$SSR_k$ for some~$k\in\{1,\dots,n-1\}$. 
Let $u^1,\dots,u^n$ be the real vectors constructed from the eigenvectors of $A$ as described and let
\[
E:=\Span\{u^1,\dots,u^{k}\} , \quad E^c:=(\R^n\setminus{E}) \cup \{0\}.
\]
Then the following properties hold:
%%%%%%%%%%%%%%%%%%%%%%%%%%%%%%
\begin{enumerate}[label=(\roman*)]
%%%%%%%%%%%%%%%%%%%%%%%%%%%%%%%%%%%%%%%%%%%%%%
\item~$\Dim(E)=k$ and~$\Dim(E^c)=n-k$; \label{it:dimnki}
\item both~$E$ and~$E^c$ are invariant under~$A$; \label{it:eec}
\item $E\subseteq \Int(P^k_-) \cup\{0\}$, and~$E^c \cap P^k_-=\{0\}$. \label{it:pkcap}
\item \label{it:convexp} There exist~$a>0$ and~$b \in (0,1)$
such that for any~$x(0)\in E  $,
$\tilde x(0) \in E^c$, with~$||x(0)||=||\tilde x(0)||=1$,  the corresponding solutions  
 of~\eqref{eq:ltis} satisfy
\be\label{eq:expco}
 ||\tilde x(j)||  \leq a  b^j ||x(j)||.
\ee
\item \label{item:lpoit}
%%%%%%
For any~$x(0)$ satisfying~
\be\label{eq:expconv}
x(0)=f+g,  \,\mbox{where}\, f\in  E \setminus\{0\}
\text{ and }  g\in E^c,
\ee
  there exists an~$q=q(x(0))\geq 0$ such that 
the corresponding solution  
 of~\eqref{eq:ltis} satisfies
$$
 s^ + (x(j))\leq k-1 \text{ for all } j\geq q. 
$$
%%
%%%%%%%%%%
%%%%%%%%%%%%%%
\end{enumerate}
%%%
\end{Theorem}

Condition \ref{item:lpoit} does not necessarily mean that $x(0)$ is an element of $E$, as it may also include some non-zero combination of the vectors $u^{k+1}, \dots,u^n$ that 
are not in $E$.  Note that assertion~\ref{item:lpoit} implies that for almost any  initial condition, the corresponding solution of the dynamical system converges to~$P^k_+$ 
in finite time.

%%%%%%%%%%%
\begin{proof}
%%%%%%%%%%%
We begin by noting that the eigenvalues of~$A$ are ordered as
\be\label{eq:eifgome}
|\lambda_1|\geq\dots \geq |\lambda_k|>|\lambda_{k+1}|\geq\dots\geq|\lambda_n|>0.
\ee

Assertion~\ref{it:dimnki} follows immediately from the fact that~$u^1,\dots,u^k$ are linearly independent. 

Pick~$z \in E \setminus\{0\}$.
Since~$\prod_{\ell=1}^k \lambda_\ell  $
 is real, either~$\lambda_{k-1},\lambda_{k}$ are both real, or
they are a complex conjugate pair. Combining this with the definition of~$E$ implies that~$z=\sum_{i=1}^k c_i v^i$, for some~$c_1,\dots, c_k$ that match~$v^1,\dots,v^k$. 
Hence,~$Az=\sum_{i=1}^k c_i \lambda_i v^i $.
Clearly, $\{c_1\lambda_1,\dots,c_k \lambda_k\}$
 also match~$\{v^1,\dots,v^k\}$,  so~$E$ is invariant under~$A$. 

 It follows from~\eqref{eq:suim} and the construction of the~$u^i$'s that~$s^+(z)\leq k-1$ for any~$z\in E\setminus\{0\}$, that is, $E\setminus \{0\} \subseteq P^k_+$. Since~$P^k_+=\Int(P^k_-)$,
we conclude that~$E\subseteq \Int(P^k_-) \cup\{0\}$.

In the remainder of the proof we consider without loss of generality 
the generic case,  where~$u^1,\dots,u^n$ are linearly independent. 
Then~$E^c=\Span\{u^{k+1},\dots,u^n\}$.  The proofs of the properties of~$E^c$ are then very similar to the proofs for~$E$, and thus we  present here only the proofs for~$E$. 

	%The definition of~$E$ and~\eqref{eq:suim} imply  that~$s^+(z) \leq k-1$, that is,~$z \in P^k_+ =\Int(P^k_-)$.   This proves~\ref{it:pkcap}.

To prove~\ref{it:convexp}, 
pick~$x(0) \in E\setminus\{0\}$ and~$\tilde x(0) \in E^c\setminus\{0\}$. Then~$x(0)=\sum_{i=1}^k c_i v^i$
and~$\tilde x(0)=\sum_{i=k+1}^n \tilde c_i v^i$, where~$c_1,\dots,c_k \in \C$ [$\tilde c_{k+1},\dots,
\tilde c_n \in \C$]
match~$v^1,\dots,v^k$ [$v^{k+1},\dots,v^n$].
Using~\eqref{eq:eifgome}, a
  straightforward argument shows  that there exists~$m>0$ such that 
\begin{align*}
 ||x(j)||&=||A^j x(0)|| \\
       &\geq m |\lambda_k|^j ||x(0)||.
\end{align*}
Similarly, there exists~$M>0$ such that~$||\tilde x(j)||\leq M  |\lambda_{k+1}|^j ||\tilde x(0)||$.
Thus,
\begin{align*}
 \frac  {||\tilde x(j)|| } {||x(j)||} & \leq \frac{M}{m} \left |
\frac { \lambda_{k+1}  } { \lambda_{k }  } \right |^j 
\frac  {||\tilde x(0)| |} {||x(0)||},
\end{align*}
%%%%%%%%%%
and combining this with~\eqref{eq:akoi} 
 proves~\eqref{eq:expco}.

To prove~\ref{item:lpoit},
pick~$x(0)  $ such that \eqref{eq:expconv} is satisfied. 
 Then~$x(0)=\sum_{i=1}^n c_i v^i$,
 where~$c_1,\dots,c_n \in \C$  
match~$v^1,\dots,v^n$, and~$\sum_{i=1}^k c_i v^i \not =0$. 
%%%
Thus,
\begin{align*}
\frac{x(j)}{||\sum_{i=1}^k c_i \lambda_i^j v^i || }&=
\frac{ \sum_{i=1}^k c_i \lambda_i^j v^i } {||\sum_{i=1}^k c_i \lambda_i^j v^i || } + \frac{\sum_{i=k+1}^n c_i \lambda_i^j v^i} {
||\sum_{i=1}^k c_i \lambda_i^j v^i||}.
%%%
\end{align*}
%%%%%%%%%
The first term on the right-hand side of this equation is a 
unit vector in~$E$, and the second term 
goes to zero as~$j\to\infty$. 
Thus, there exists~$r\geq 0$ such that~$x(r) \in P^k_-$. 
Then~$x(r+1)\in P^k_+$,
 and the invariance of~$P^k_+$ implies that
$x(j)\in P^k_+$ for all~$j\geq r+1$. 
%%%%%%%%%
\end{proof}
%%%%%%%%
%%%%%%%%%%%%%%%%%%%%%%%%%%%%%%%%%%%%%%%%%%%%%%%
\subsection{Dynamics of exterior products   }
%%%%%%%%%%%%%%%%%%%%%%%%%%%%%%%%%%%%%%%%%%%%%%%
Recall that if~$Z\in \R^{n\times k}$, 
with columns~$z^1,\dots,z^k \in \R^n$,
 then its~$k$th multiplicative compound 
$Z^{(k)} \in \R^{\binom{n}{k} \times \binom{n}{k} }$
is the exterior  product~$z^1 \wedge \dots \wedge z^k$, represented as a column vector~\cite{muldo1990}. 
For example, for~$z^1=\begin{bmatrix} r_1&r_2&r_3\end{bmatrix}^T$ 
and~$z^2=\begin{bmatrix} w_1&w_2&w_3\end{bmatrix}^T$, we have 
\begin{align*}
%%%%%%%%%%%%%%%%%%%%%%%%%
Z^{(2)}&= \begin{bmatrix}r_1 &w_1 \\ 
r_2 &w_2 \\
r_3 &w_3  \end{bmatrix} ^{(2)} \\
%%%%%%  
&=\begin{bmatrix}
 r_1 w_2-r_2 w_1 & r_1 w_3-r_3 w_1& r_2 w_3-r_3 w_2   
\end{bmatrix}^T.
\end{align*}

Consider  the dynamics~\eqref{eq:ltis},
where~$A\in\R^{n\times n}$ is~$SSR_k$, and pick~$k$ initial conditions~$w^1,\dots,w^k \in \R^n$. Let
\be\label{eq:decfCC}
X(j):=\begin{bmatrix} x( j,w^1) &\dots& x(j,w^k) \end{bmatrix} \in\R^{n\times k}. 
\ee
Then~$X(j+1)=A X(j)$. Taking the~$k$th multiplicative compound on both sides of this equation and using~\eqref{eq:mulcom}
yields
\be\label{eq:etaj}
		\eta(j+1)		=A^{(k)} \eta(j)	,
\ee
where
\be\label{eq:defetaj}
\eta(j):=x( j,w^1) \wedge\dots\wedge  x(j,w^k).
\ee
The magnitude of this wedge product 
is the    volume   of the $k$-dimensional parallelotope whose edges are the given vectors.

%%%%
%%%%%
\begin{Example}
%%%
Suppose that~$n=3$,~$A:=\begin{bmatrix} \lambda_1& 0&0 \\0&\lambda_2&0 \\ 0&0&\lambda_3 \end{bmatrix}$,~$k=2$,~$w^1=e^p $ and~$w^2=e^q$ for some~$p,q\in \{1,2,3\}$. 
Then 
\begin{align*}
%%%
							\eta(j)&=x(j,e^p)\wedge x(j,e^q) \\
							    &=(\lambda_p^j e^p)\wedge ( \lambda_q^j e^q)\\
									&=\lambda_p ^j \lambda_q^j  (e^p\wedge   e^q)\\
									&=(\lambda_p   \lambda_q)^j \eta(0).
%%%
\end{align*}
This implies that under the dynamics~\eqref{eq:ltis}
 the unsigned area of the  
 parallelogram having~$e^p$ and~$e^q$ 
 as two of its sides  scales as~$(\lambda_p   \lambda_q)^j$.
On the 
other-hand,~$A^{(2)} = \begin{bmatrix} \lambda_1 \lambda_2& 0&0  \\
0&\lambda_1 \lambda_3& 0  \\
0&0 &  \lambda_2 \lambda_3 \end{bmatrix}$.
%%%
\end{Example}

If~$A$ is~$SSR_k$ then  either
every entry of~$B:=A^{(k)}$ is positive or negative. 
We assume that~$B>0$ (the case~$B<0$ can be treated similarly). By the Perron Theorem, 
the spectral radius of~$B$, denoted~$\rho(B)$,
 is a positive eigenvalue
and there exist positive vectors~$v^B,w^B$,
such that~$B v^B=\rho(B) v^B$ and~$B^T w^B=\rho(B) w^B$. 
By normalization, we may assume that~$(v^B)^T w^B=1$.
Then furthermore
\be\label{eq:jlimit}
 \lim_{j\to \infty} \left (\frac{B}{\rho(B)} \right )^j=v^B (w^B)^T 
\ee
(see, e.g.,~\cite[Chapter~8]{matrx_ana}). 
This yields the following result.
\begin{Lemma}\label{lem:pf}
%%%
Suppose that~$A$ is~$SSR_k$ and that~$B:=A^{(k)}>0$.
 Pick~$k$ initial conditions~$w^1,\dots,w^k \in \R^n$,
and define~$X(j)$ and~$\eta(j)$ as in
\eqref{eq:decfCC} and~\eqref{eq:defetaj}. Then
\be\label{eq:likjp}
\lim_{j\to \infty}   \frac{\eta(j)}{(\rho(B) ) ^j} = (w^B)^T \eta(0) v^B.  
\ee
 %%%
\end{Lemma}

\begin{proof}
	%%%%%%%%%%%%%%%%%%
	By~\eqref{eq:etaj}, 
	$\eta(j)		=B^j \eta(0)$,  i.e., $\frac{\eta(j)	}{(\rho(B))^j}	=\left(\frac{B}{\rho(B)}\right)^j \eta(0)$.
	Taking~$j\to\infty$ and using~\eqref{eq:jlimit} completes the proof. 
\end{proof}
\begin{Remark}
%%%
Suppose that the spectral radius~$\rho(A)$ of~$A$ satisfies $ 
 \rho(A)<1$. Then~$\lim_{j\to \infty} A^jx =0$ for all~$x\in\R^n$ and thus
\[
\eta(j) =(A^j w^1) \wedge\dots\wedge  (A^j  w^k),
\]
satisfies~$\lim_{j\to \infty} \eta(j)=0$.  Since every eigenvalue of~$A^{(k)}$ is the product of~$k$ eigenvalues of~$A$, in this case~$\rho(B)<1$
so~\eqref{eq:likjp} also shows that~$\eta(j)$
goes to zero as~$j\to\infty$.
\end{Remark}

%%%%%%%%%%%%%%%%%%%%%%%%%%%%%%%%%%%%%%
\begin{Example} 
Consider the case~$n=3, k=2$,
$$A:= \begin{bmatrix} 0.79& 0.2& 0.01\\
 0.1& 0.8& 0.1\\
 0.01& 0.1& 0.89\end{bmatrix},$$ $w^1=e^1$, and~$w^2=e^2$. 
In other words, we consider the evolution of  
 the unsigned area of the  
 parallelogram with~$e^1$ and~$e^2$ 
 as two of its sides.
A calculation yields
\[
B:=A^{(2)}=\begin{bmatrix}
 0.612& 0.078& 0.012 \\0.077& 0.703& 0.177\\ 0.002& 0.088& 0.702
\end{bmatrix}
\]
 (so~$A$ is~$SSR_2$),
$\rho(B)= 0.8430$,
$$v^B=\begin{bmatrix} 0.2991 & 0.8075& 0.5084 \end{bmatrix}^T,$$
and~$$w^B=\begin{bmatrix} 0.2203 & 0.6394 & 0.8217 \end{bmatrix}
^T,$$ (note that~$(w^B)^T v^B=1$). 
We  compute~$\eta(15)$ in two different ways. 
First,
\begin{align}\label{eq:poity}
\eta(15)&=(A^{15}e^1)\wedge (A^{15}e^2)\nonumber \\
&=\begin{bmatrix} 0.2397& 0.2190& 0.1858 \end{bmatrix}^T    \nonumber \\& \; \; \wedge
 \begin{bmatrix} 0.4228&0.4103& 0.3859 \end{bmatrix}^T \nonumber \\
&=0.0057  e^1+ 0.0139  e^2+0.0083 e^3.
\end{align}
%%%%
Second, it follows from~\eqref{eq:likjp} that
%%%%
\begin{align*}
%%%
\eta(15)&\approx (\rho(B))^{15} (w^B)^T\eta(0) v^B\\
&=(\rho(B))^{15} w^B_1 v^B\\
&=\begin{bmatrix}  0.0051& 0.0137& 0.0086 \end{bmatrix}^T,
%%%%%
\end{align*} 
and this is indeed an approximation of~\eqref{eq:poity}.
\end{Example}

\section{Discussion}
%%%%%%%%%%%%%%%%%%%%%%%%%%%%%%%%%%%%%%

Positive  systems and their nonlinear counterpart of monotone systems form a class of dynamical systems of fundamental importance in systems
biology, neuroscience, and bio-chemical networks, and has recently also found important
applications in control engineering for large-scale systems~\cite{RANTZER201572}. 

We introduced a new class of DT linear systems that  generalize  the important notion of DT positive linear systems. 
Such systems map the set of vectors with up to~$k-1$ sign variations to itself.

An interesting research direction is to study DT nonlinear systems whose variational equation is a~$k$-positive linear  system. Since the variational equation~\eqref{eq:vari} includes the integral of a matrix, this raises the following question: when is the integral of a matrix~$SSR_k$?

Lemma~\ref{lem:pf}
 describes a convergence to a ray   for the exterior product.
 We believe that this can generalized to the~DT
time-varying linear system~\eqref{eq:dts}, with the matrices $A(i)$ taken from a compact set, 
 using the Birkhoff-Hopf  theory~\cite{eveson_nussbaum_1995}. 

Another interesting research direction 
 may be the extension of~$k$-positive systems   to~DT control systems as was done  
for~CT monotone systems in~\cite{mcs_angeli_2003}.
Finally, our results highlight the importance of an efficient algorithm for 
determining if a given matrix is~$SSR_k$ for some~$k$. This issue is currently under study~\cite{paperLA}.

 %%%%%%%%%%%%%%%%%%%%%%%%%%%%%%%corollary%%%%%%%%%%%%%%%%%%%%%%%%%
 \begin{comment}
 %%%%%%%%%%%%%%%%%%%%%%%%%%%%%%
 MATLAB CODE  - PLS DON'T DELETE
 %%%%%%%%%%%%%%%%%%%%%%%%%%%%%
 hold off;
 A=[ 9 2 -2 1 ; 3 10 1 -1; -4 1.5 12 4 ; 1 -1 2 15];
 x0= [ 1  1 -1 1 ]';
 xk=x0;tsp=20;
 for k=0:1:tsp
 num=0;
 if xk(1)*xk(2)*xk(3)*xk(4) == 0
 print('a zero entry-check s^+') 
 pause
 end
 
 if xk(1)*xk(2)*xk(3)*xk(4) ~= 0
 % no zero entries
 if xk(1)*xk(2)<0 num=num+1; end
 if xk(2)*xk(3)<0 num=num+1; end
 if xk(3)*xk(4)<0 num=num+1; end
 xk=A*xk;
 plot(k,num,'ok','MarkerSize',10);hold on;
 end
 end
 grid on;
 xlabel('$j$','interpreter','latex', .....
 'fontsize', 16);
 aa=axis;
 axis([ aa(1) aa(2) 0  aa(4)]);
 %%%
 %%%
 \end{comment}

% Generated by IEEEtranS.bst, version: 1.14 (2015/08/26)

%%%%%%%%%%%%%%%%%%
\end{document}